\documentclass[11pt,reqno]{amsart}
\usepackage{amsmath,amssymb,mathrsfs}
\usepackage{graphicx,cite}
\usepackage{color}
\usepackage{hyperref}
\usepackage{caption}
\usepackage{subcaption}
\usepackage{todonotes}

\newtheorem{theorem}{Theorem}[section]
\newtheorem{corollary}[theorem]{Corollary}
\newtheorem{lemma}[theorem]{Lemma}

\newtheorem{remark}[theorem]{Remark}

\newcommand{\sP}{{s_{\text{PCA}}}}
\newcommand{\sD}{{s_{\text{DMD}}}}

\newcommand{\D}{\mathcal{D}}
\newcommand{\G}{\mathcal{G}}

\newtheorem{inverseproblem}{Inverse problem}
\newtheorem{exam}{Example}

\numberwithin{equation}{section}

\allowdisplaybreaks[4]

\title[inverse random source problem]{A data-assisted two-stage method for the inverse random source problem}

\author{Peijun Li}
\address{Department of Mathematics, Purdue University, West Lafayette, Indiana 47907, USA}
\email{lipeijun@math.purdue.edu}

\author{Ying Liang}
\address{Department of Mathematics, Purdue University, West Lafayette, Indiana 47907, USA}
\email{liang402@purdue.edu}

\author{Yuliang Wang}
\address{Research Center for Mathematics, Beijing Normal University, Zhuhai 519087, China; BNU-HKBU United International College, Zhuhai 519807, China}
\email{yuliangwang@uic.edu.cn}

\thanks{The first author was supported in part by the NSF grant DMS-2208256.}



\begin{document}

\maketitle

\begin{abstract}
We propose a data-assisted two-stage method for solving an inverse random source problem of the Helmholtz equation. In the first stage, the regularized Kaczmarz method is employed to generate initial approximations of the mean and variance based on the mild solution of the stochastic Helmholtz equation. A dataset is then obtained by sampling the approximate and corresponding true profiles from a certain a-priori criterion. The second stage is formulated as an image-to-image translation problem, and several data-assisted approaches are utilized to handle the dataset and obtain enhanced reconstructions. Numerical experiments demonstrate that the data-assisted two-stage method provides satisfactory reconstruction for both homogeneous and inhomogeneous media with fewer realizations.
\end{abstract}

\section{Introduction}

Inverse source problems are of great importance in many fields, such as antenna synthesis, medical imaging, and earthquake monitoring. These problems involve determining the unknown sources that generate prescribed radiated wave patterns. Over the past four decades, they have been extensively studied. However, due to the presence of nonradiating sources, inverse source problems generally do not have unique solutions when only boundary measurements at a fixed frequency are used \cite{bleistein1977nonuniqueness, devaney1982nonuniqueness}. 
 In order to tackle the difficulty of these issues, attempts have been made to find the least energy solution \cite{marengo1999inverse}. The use of multifrequency data has been employed to guarantee uniqueness and improve the stability of the problem \cite{bao2015multi}.

In many applications, it is desirable to take into account the uncertainties that may arise from the unpredictability of the surrounding environment, incomplete information of the system, or random noise of the measurements. These inverse problems involving such uncertainties are called stochastic inverse problems. A number of theoretical studies for stochastic inverse scattering problems can be found in \cite{fouque2007wave,helin2017inverse,badieirostami2010wiener}. The inverse random source problem seeks to identify the mean and variance of random sources from the data they produce. In \cite{bao2014inverse}, a computational framework was proposed for the one-dimensional problem by using the inverse sine or cosine transform to reconstruct the mean and variance of the random source. For higher dimensional problems, the regularized block Kaczmarz and eigenfunction-based methods were developed in \cite{bao2016inverse} and \cite{li2017inverse} for the inverse source problem in a homogeneous and inhomogeneous medium, respectively. In \cite{li2017stability}, a stability analysis was carried out for the inverse random source problem of the stochastic Helmholtz equation driven by the white noise. Uniqueness was examined in \cite{li2020inverse, li2022inverse} for the inverse source problem where the random source was modeled as a generalized microlocally isotropic Gaussian random field. The existing numerical methods for the inverse random source problem have several drawbacks: they require a large number of realizations of the stochastic direct problem; the mean and variance of the source must be continuous for satisfactory reconstruction; the medium is assumed to be either homogeneous or inhomogeneous but known. In practice, the inhomogeneous medium may be unknown, which limits the feasibility of existing numerical methods. Moreover, the accuracy of the reconstruction is affected by the number of realizations of the stochastic direct problem and the continuity of the mean and variance of the source.

Recently, machine learning techniques have been successfully applied to a wide range of scientific areas, particularly for solving classification and segmentation problems. Motivated by these successes, machine learning algorithms are being applied to solve inverse problems modeled by parametric partial differential equations (PDEs). We briefly introduce two common existing strategies based on neural networks. The first strategy utilizes neural networks as surrogates for differential equations, such as the physics informed neural network (PINN) \cite{chen2020physics}, the Deep Ritz method \cite{yu2018deep}, and the weak adversarial neural network (WAN) \cite{zang2020weak}. Algorithms are developed to search for unknown parameters in the PDE model and to obtain numerical solutions over the approximation space consisting of neural networks. However, re-training the neural network for a different problem governed by the same PDE is quite expensive. The second main strategy utilizes the training dataset in the form of paired observations and parameters to create a model that can be used to make predictions. In such algorithms, the observations are treated as input and the corresponding exact parameters are treated as the output of a given neural network architecture. The neural network is trained to learn the mapping from the observations to the parameters, as demonstrated in \cite{adler2017solving,  zhu2018bayesian}. This strategy has been used for data-driven model discovery \cite{bongard2007automated, schmidt2009distilling} to reconstruct governing equations from observed data. Besides these two strategies, there is an emerging body of work on operator learning. In \cite{maarten2022deep}, the authors propose the operator recurrent neural network to directly learn the inverse operator between infinite-dimensional spaces. There has also been interest in learning regularizers for inversion, and a comprehensive account of this topic can be found in \cite{arridge2019solving}.

There have been a number of works on machine learning algorithms developed for solving inverse scattering problems. In \cite{khoo2019switchnet}, a neural network model with a SwitchNet architecture was proposed and trained to learn the direct and inverse maps between the scatterers and the scattered wave field. In \cite{uhlmann2020convolutional}, the authors incorporated the physical model of wave propagation into a deep neural network architecture to solve the inverse problem related to nonlinear wave equations. In \cite{gao2022artificial}, a fully connected neural network was trained to learn the mapping between the limited-aperture radar cross-section data and the Fourier coefficients of the unknown scatterers.

In this work, we propose a data-assisted two-stage method for solving the inverse random source problem. The goal is to achieve satisfactory reconstructions for discontinuous mean and variance of the random sources with fewer realizations, even in an unknown inhomogeneous medium. In the first stage, we utilize the regularized block Kaczmarz method to obtain initial approximations by using the mild solution of the stochastic Helmholtz equation in a homogeneous medium. In the second stage, several machine learning techniques are applied to learn the mapping from the approximations to the exact parameters, thus enabling more accurate predictions. In this way, the second stage is formulated as an image-to-image translation problem, which aims to transfer images between source and target domains while preserving content representations. Our approach has two salient features. Firstly, partial physical information is embedded in the approximation from the first stage, thus reducing the computational burden on the neural network in the second stage. Secondly, the approximation is of the same dimension as the discretized statistics of the source, allowing for the use of various advanced machine learning algorithms for image-to-image translation problems. Similar ideas have been adopted in \cite{zhou2020improved, xu2020deep, chen2020review} for deterministic inverse scattering problems and in \cite{guo2021construct} for diffusive optical tomography.

Machine learning algorithms have found wide application in solving image-to-image translation problems, including in computer vision, medical imaging, and photo enhancement. In our work, we carry out a comparative study on the performance of four different machine learning techniques in the second stage of the experiment. The first method utilizes principal component analysis (PCA) to reduce the dimensionality of the training data, followed by regression to produce the reconstructed statistics of the random source. The second method approximates the image-to-image translation problem as a linear dynamical system using the dynamic mode decomposition (DMD) approach \cite{schmid2010dynamic}. The other two methods are based on convolutional neural networks. The first approach utilizes the U-Net architecture \cite{ronneberger2015u}, a supervised learning approach originally designed for image segmentation. The second algorithm, pix2pix \cite{isola2017image}, is an advanced method for image-to-image translation based on a conditional generative adversarial neural network (cGAN). Numerical experiments have been conducted extensively to verify that these machine learning algorithms can effectively improve the reconstruction of the random source.

The paper is outlined as follows. In Section \ref{Sec:2}, the mathematical model for the inverse random source problem is introduced. In Section  \ref{Sec:3}, we present a reconstruction scheme for the inverse problem based on integral equations and regularized block Kaczmarz method, and provide an error analysis of the method. In Section  \ref{Sec:4}, we detail the two-stage approach, which involves obtaining an approximation of the statistics using the reconstruction scheme in the first stage and enhancing it with machine learning algorithms in the second stage. Finally, we present numerical examples to demonstrate the performance of the proposed algorithm in Section  \ref{Sec:5}.

\section{Problem Statement}\label{Sec:2}

Consider the two-dimensional Helmholtz equation driven by a random source 
\begin{equation}\label{eq:forward}
    \Delta u(x,\kappa) +\kappa^2 (1+\eta(x)) u(x,\kappa)  = f(x),\quad x\in \mathbb{R}^2,
\end{equation}
where the wavenumber $\kappa$ is a positive constant, $\eta$ is a deterministic function representing the relative permittivity of the medium, and the random source $f$ is assumed to be given by  
\[
f(x) = g(x) +h(x) \dot{W}_x
\]
with $g$ and $h \geq 0$ being two real-valued functions and $\dot{W}_x$ being the spatial white noise. 
Additionally, we assume that $g$ and $h$  are compactly supported in a rectangular domain $D\subset \mathbb{R}^2$, where $\overline{D} \subset B_R = \{x\in\mathbb R^2: |x|<R\}$ for some $R>0$, and $\eta$ is compactly supported in \( B_R \). As usual, the scattered field satisfies the Sommerfeld radiation condition
\begin{equation}
  \label{eq:Sommerfeld}
    \lim_{r=|x|\to \infty} \sqrt{r} (\partial_r u - {\rm i} \kappa u )= 0.  
\end{equation}

There are two problems associated with the model \eqref{eq:forward}--\eqref{eq:Sommerfeld}. The direct problem is to study the radiated random wave field $u$ for a given random source  $f$. This work focuses on the inverse random source problem, which aims at determining the mean $g$ and the variance $h^2$ of the random source $f$ from the radiated wave field $u$ that is measured on the boundary $\Gamma_R=\{x\in\mathbb R^2: |x|=R\}$ at a discrete set of wavenumbers $\kappa_j$, $j=1,2,...,m$. More precisely, we consider the following three types of inverse source problems according to the status of $\eta$.

\begin{inverseproblem}[IP1]\label{Homo}
The medium is homogeneous with $\eta=0$. The inverse random source problem is to recover $g$ and $h^2$ from the wave field $u$ measured on the boundary $\Gamma_R$.
\end{inverseproblem}

\begin{inverseproblem}[IP2]\label{Inhomo}
The medium is inhomogeneous with known $\eta$. The inverse random source problem is to recover $g$ and $h^2$ from the wave field $u$ measured on the boundary $\Gamma_R$.
\end{inverseproblem}

\begin{inverseproblem}[IP3]\label{UnknownInhomo}
The medium is inhomogeneous with unknown  $\eta$. The inverse random source problem is to recover $g$ and $h^2$ from the wave field $u$ measured on the boundary $\Gamma_R$.
\end{inverseproblem}

Previous studies have focused on solving the inverse source problems IP1 and IP2. Based on Green's function, IP1 was studied in \cite{bao2016inverse}. In particular, the direct source problem was shown to have a unique mild solution, and a regularized block Kaczmarz method was developed for the inverse source problem. When the medium is inhomogeneous with a variable function $\eta$, the explicit Green function is no longer available. In \cite{li2017inverse}, the authors examined the Lippman--Schwinger integral equation for the direct source problem and considered the Dirichlet eigenvalue problem of the corresponding Helmholtz equation for IP2. However, these methods fail to address IP3, where the inhomogeneous medium is unknown. 

This study introduces an innovative data-driven approach suitable for solving all three inverse source problems mentioned previously. The training process of the method consists of two stages, which we abstractly describe as follows.

\begin{description}
\item[Stage one] The training dataset consists of  $M$ samples, denoted as  $ \{ (p_i, Y_i) \}_{i=1}^M $, where each  $p_i$ represents multi-frequency measurements of the wave field, and each  $Y_i$ represents the corresponding statistics of the source for the  $i$-th sample. We apply a regularized block Kaczmarz method using measurement  $p_i$, which will be further explained in Section \ref{Sec:3}, leading to an initial approximation $X_i$.

\item[Stage two] Form a new dataset $ \{ (X_i, Y_i) \}_{i=1}^M $ and employ learning algorithms to train a model $\chi$ that approximates the mapping from the initial approximation $X_i$ in stage one to the ground truth $Y_i$.
\end{description}

Having obtained the model $\chi$ in stage two, we can use it to reconstruct source statistics with any wave measurement $p_{\text{test}}$ in two steps. Firstly, by using the regularized Kaczmarz method as in the training process, we obtain an initial approximation $X_{\text{test}}$ from the measurement data  $p_{\text{test}}$. Secondly, we input  $X_{\text{test}}$ into the model  $\chi$ to obtain a reconstruction  $\chi(X_{\text{test}})$, which is expected to provide a better approximation of the exact statistics  $Y_{\text{test}}$ than the initial approximation $X_{\text{test}}$.

The training dataset in stage one can be obtained from physical experiments or real-world applications. In this work, we generate the synthetic data by creating random samples of $Y_i$ and solving the direct problem numerically for the measurement data $p_i$. The details for generating the training dataset can be found in Section \ref{Sec:5}.  The function $\eta$ is only used when computing the synthetic data, and the information of the medium is not needed in the training or testing process of the proposed method.  Thus, the two-stage approach works for IP1, IP2, and IP3, regardless of the inhomogeneity function $\eta$.

\section{Stage one: the Kaczmarz method}\label{Sec:3}

Stage one is to generate an initial image of the reconstruction. In this section, we briefly introduce the mild solution to the direct source problem of the stochastic Helmholtz equation in a homogeneous medium and the regularized block Kaczmarz method for solving the inverse random source problem. The detail can be found in \cite{bao2016inverse}.

\subsection{Integral equations}

Following \cite[Hypothesis 2.4]{bao2016inverse}, we assume that $g\in L^2(D)$, $h\in C^{0,\alpha}(D)$,  where $\alpha\in(0,1]$, and $h\in L^\beta(D)$, where $\beta\in(\beta_0,\infty]$ if $\frac{3}{2}\leq \beta_0\leq 2$ or $\beta\in (\beta_0, \frac{3\beta_0}{3-2\beta_0})$ if $1\leq \beta_0\leq \frac{3}{2}$. It is shown in \cite[Theorem 2.7]{bao2016inverse} that the direct source problem \eqref{eq:forward}--\eqref{eq:Sommerfeld} admits a unique mild solution 
\begin{equation}\label{eq:mildsolution}
u(x) = \int_D G(x,y,\kappa) g(y) dy +\int_D G(x,y,\kappa) h(y) dW_y,\quad a.s.,
\end{equation}
where $G(x,y)=-\frac{\rm i}{4} H_0^{(1)} (\kappa |x-y|)$ is Green's function of the two-dimensional Helmholtz equation with $H_0^{(1)}$ being the Hankel function of the first kind with order zero.

Taking the expectation on both sides of \eqref{eq:mildsolution}, we get 
\begin{equation*}
\mathbb E(u(x,\kappa_j) ) = \int_D G(x,y,\kappa_j) g(y)dy. 
\end{equation*}
Splitting the above equation into the real and imaginary parts leads to 
\begin{align}
    \mathbb E(\Re u(x,\kappa_j)) & =  \int_D \Re G(x,y,\kappa_j) g(y)dy, \label{eq:meanreal}\\
       \mathbb E(\Im u(x,\kappa_j))& =  \int_D \Im G(x,y,\kappa_j) g(y)dy. \label{eq:meanimag}
\end{align}
The mean $g$ of the source $f$ can be obtained by solving either \eqref{eq:meanreal} or \eqref{eq:meanimag}. We use \eqref{eq:meanreal} to present the results in the numerical experiments. 

Taking the variance on both sides of the real and imaginary parts of \eqref{eq:mildsolution} yields
\begin{align}
    \mathbb V (\Re u(x,\kappa_j)) & = \int_D |\Re G(x,y,\kappa_j)|^2 h^2(y)dy, \label{eq:varreal}\\
    \mathbb V (\Im u(x,\kappa_j)) & =  \int_D |\Im G(x,y,\kappa_j)|^2 h^2(y)dy. \label{eq:varimag}
\end{align}
Subtracting \eqref{eq:varimag} from  \eqref{eq:varreal}, we obtain 
\begin{equation}\label{eq:var}
     \mathbb V (\Re u(x,\kappa_j)) - \mathbb V (\Im u(x,\kappa_j)) =  \int_D \left(|\Re G(x,y,\kappa_j)|^2- |\Im G(x,y,\kappa_j)|^2\right) h^2(y)dy. 
\end{equation}
In accordance with \cite{bao2016inverse}, \eqref{eq:var} yields better reconstruction results than \eqref{eq:varreal} or \eqref{eq:varimag}, since the singular values of the discretized integral kernel in \eqref{eq:var} decay more slowly than those in \eqref{eq:varreal} or \eqref{eq:varimag}. Therefore, \eqref{eq:var} is adopted to reconstruct $h^2$ in our numerical experiments.

Neglecting the discretization error, the matrix equations corresponding to integral equations \eqref{eq:meanreal} and \eqref{eq:var} can be formally expressed as
\begin{equation}\label{operatoreq}
    A_j q = p_j, \quad j=1,\dots, m,
\end{equation}
where $p_j$ denotes the data $\mathbb{E}(\Re u(x,\kappa_j))$ or  $\mathbb{V}(\Re u(x,\kappa_j))- \mathbb{V}(\Im u(x,\kappa_j))$ at measurement points on $\Gamma_R$, $q$ stands for the unknown statistics $g$ or \( h^2 \) taken at sampling points in $D$, and $A_j$ represents the matrix arising from discretizing the integral kernels in \eqref{eq:meanreal} and \eqref{eq:var}. To solve the ill-posed equation \eqref{operatoreq}, a regularized block Kaczmarz method is introduced in \cite{bao2016inverse}: let $q_\gamma^0 = 0$,
\begin{equation}\label{KMscheme}
    \begin{cases}
    q_0 = q_\gamma^k,\\
    q_j =  q_{j-1}+A_j^* (\gamma I +A_j A_j^*)^{-1}(p_j-A_j q_{j-1}),\quad j=1,\dots, m,\\
    q_\gamma^{k+1} = q_m,
    \end{cases}
\end{equation}
where $k=0,1,\dots$ is the iteration index of the outer loop, $A_j^*$ denotes the conjugate transpose of $A_j$,  $\gamma>0$ is the regularization parameter, and $I$ is the identity matrix.

Combined with multi-frequency data, the regularized Kaczmarz method effectively provides reconstructions for the statistics of the random source. However, the method has several limitations. 

\begin{enumerate}
    \item Accurately approximating the mean or variance using their corresponding averages in \eqref{eq:meanreal} or \eqref{eq:var} typically requires a large number of trials. In practice, only a limited number of realizations are available,  leading to significant errors in the approximation on the left-hand side of \eqref{eq:meanreal} or \eqref{eq:var}. Figure \ref{fig:compare} shows the reconstructions of the mean of a random source using different numbers of realizations. The accuracy of the reconstruction improves as more realizations are used for the reconstruction.
    \item The sequence generated by the Kaczmarz method typically converges to a minimum norm solution of the linear system as demonstrated by \cite{natterer86} . Consequently, it is difficult for the method to reconstruct a random source with discontinuous statistics, particularly for the regularized method. Figure \ref{fig:compare} displays a random source with a piecewise constant mean. Although increasing the number of realizations improves the result,  accurately capturing the edges of the true mean remains challenging.
    \item The mild solution \eqref{eq:mildsolution} and the Fredholm integral equations \eqref{eq:meanreal}--\eqref{eq:var} are based on the assumption that the medium is homogeneous. However, this method cannot be directly applied to solve IP2 and IP3, where the medium is inhomogeneous and the analytical expression of the Green function is not available. To address IP2, Li et al. \cite{li2017inverse} suggested an approach utilzing the corresponding Dirichlet eigenvalue problem of the Helmholtz equation. However, this approach is not applicable to complex media and can be time consuming when higher order eigenfunctions are needed to represent sources with higher frequency modes. Moreover, it cannot be applied to solve IP3, where the inhomogeneous medium is assumed to be unknown.
\end{enumerate}

\begin{figure}[ht!]
  \centering
  \begin{subfigure}[b]{0.24\textwidth}
    \centering
    \includegraphics[width=\textwidth]{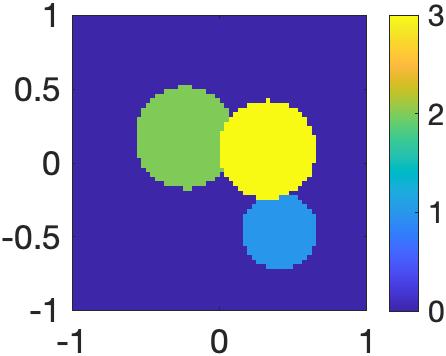}
    \caption{}
  \end{subfigure}
  \hfill
  \begin{subfigure}[b]{0.24\textwidth}
    \centering
    \includegraphics[width=\textwidth]{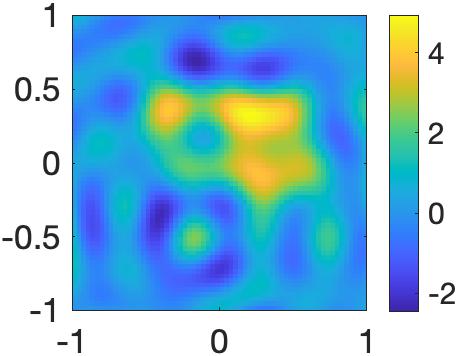}
    \caption{}
  \end{subfigure}
  \hfill
  \begin{subfigure}[b]{0.24\textwidth}
    \centering
    \includegraphics[width=\textwidth]{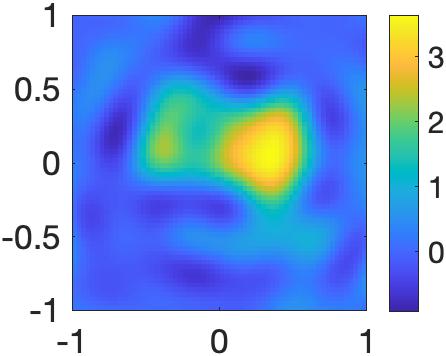}
    \caption{}
  \end{subfigure}
  \hfill
  \begin{subfigure}[b]{0.24\textwidth}
    \centering
    \includegraphics[width=\textwidth]{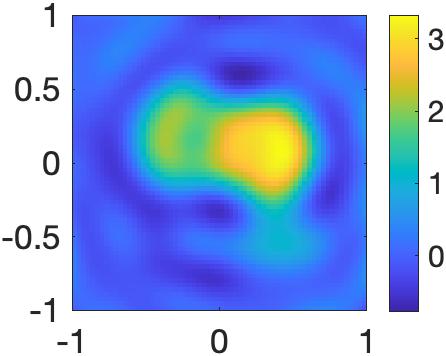}
    \caption{}
  \end{subfigure}
  \caption{Reconstructions with different numbers of trials. (A) the exact mean of a random source;
(B/C/D) The reconstruction of mean with measurements of 10/100/1000 realizations.}  \label{fig:compare}
\end{figure}

In this work, we adopt the regularized Kaczmarz method \eqref{KMscheme} as the first stage to generate an initial image of the reconstruction, independent of the relative permittivity $\eta$. Subsequently, the initial approximation is then passed to the second stage of the image-to-image translation problem, which can be effectively handled by machine learning techniques.

\subsection{Error analysis}

In this section, we present an error analysis for the regularized Kaczmarz method \eqref{KMscheme}. The total error is composed of two parts: the expectation or variance in the expression of data $p_j$ must be approximated by the sample mean, and the measurement of the wave field $u(x,\kappa_j)$ may be polluted by random noise. In \cite{bao2017inverse}, the convergence of the regularized Kaczmarz method \eqref{KMscheme} was shown for noise-free data. However, since data always contains noise in practice, an accumulated error may be generated during the iteration, leading to the semi-convergence phenomenon, which was examined in \cite{elfving2014semi} for the Kaczmarz method. For a comprehensive account of the general Kaczmarz method for solving a linear system of equations, we refer to \cite{natterer86}. 

We provide a theoretical analysis of the semi-convergence property of the regularized block Kaczmarz method. Recall that the integral equation \eqref{eq:meanreal} or \eqref{eq:var} for the reconstruction of $g$ or $h^2$ can be written as a matrix equation system \eqref{operatoreq}:
\begin{equation*}
    A_j q = p_j, \quad j=1,\dots, m.
\end{equation*}
The system \eqref{operatoreq} is equivalent to
\begin{equation*}
Aq = p,
\end{equation*}
where
$$A = \left[\begin{array}{c} A_1\\ \vdots\\ A_m \end{array}\right],\quad p = \left[\begin{array}{c} p_1\\ \vdots\\ p_m \end{array}\right].$$

In the inverse source problem, there are usually fewer measurement points compared to the number of sampling points of the solution. As a result, it is reasonable to assume that the matrix $A$ has full row rank. Let $A^\dagger$ denote the pseudo-inverse of $A$ and $q^\dagger = A^\dagger p$ denote the minimal norm solution of the inverse source problem with the noise-free data $p$. Let $\tilde{p}= [ \tilde{p}_1^*,  \tilde{p}_2^*,..., \tilde{p}_m^*]^*$ denote the perturbed data, and $q_\gamma^k$, $\tilde{q}_\gamma^k$ denote the outputs of the scheme $\text{\eqref{KMscheme}}$ with $p$ and $\tilde{p}$ as inputs respectively, after $k$ iterations.

 The error between the minimal norm solution $ q^\dagger$ and the \( k \)-th iteration $\tilde{q}_\gamma^k$  can be decomposed into two terms:
\begin{align*}
  \tilde{q}_\gamma^k - q^\dagger = (\tilde{q}_\gamma^k - q_\gamma^k ) + (q_\gamma^k - q^\dagger).
\end{align*}
According to \cite[Theorem 3.7]{bao2017inverse}, we have $q_\gamma^k - q^\dagger \rightarrow 0$ as $k \to \infty$ and $\gamma \to 0$.
In the following, we analyze \( \tilde{q}_\gamma^k - q_\gamma^k \), which represents the data error accumulated to the \( k \)-th iteration. To this end, we decompose
\[
\gamma I+A A^{* }=\left[\begin{array}{ccc}
\gamma I+A_1 A_1^{* } & \cdots & A_1 A_m^{* } \\
\vdots & \ddots & \vdots \\
A_m A_1^{* } & \cdots & \gamma I+A_m A_m^{* }
\end{array}\right]=D_\gamma+L+L^{* },
\]
where the matrices $D_\gamma$ and $L$ are given by 
$$
D_\gamma=\left[\begin{array}{ccc}
\gamma I+A_1 A_1^{* } & \cdots & 0 \\
\vdots & \ddots & \vdots \\
0 & \cdots & \gamma I+A_m A_m^{* }
\end{array}\right], \quad L=\left[\begin{array}{cccc}
0 & 0 & \cdots & 0 \\
A_2 A_1^{* } & 0 & \cdots & 0 \\
\vdots & \ddots & \ddots & \vdots \\
A_m A_1^{* } & \cdots & A_m A_{m-1}^{* } & 0
\end{array}
\right].
$$
Let 
\begin{equation}
 M_\gamma = (D_\gamma +L)^{-1}, \quad  Q_\gamma = A^* M_\gamma A,\quad B_\gamma  = I-Q_\gamma,\label{notation}
 \end{equation}
the regularized block Kaczmarz method   \eqref{KMscheme} can be equivalently formulated as
\begin{equation}\label{equivKM}
q_\gamma^{k+1} = B_\gamma q_\gamma^k + A^*  M_\gamma p.
\end{equation}
To estimate the error term \( \tilde{q}_\gamma^k - q_\gamma^k \), we define
 \begin{equation}\label{notationBR}
 \tilde{B}_\gamma = B_\gamma P,\quad  R_\gamma = (D_{2\gamma})^{1/2} M_\gamma A ,
 \end{equation}
 where $P$ is the orthogonal projection onto the range space of $A^*$, i.e., 
$\mathcal{R}(A^*)$, and \( D_{2\gamma} \) denotes the matrix \( D_\gamma \) with \( \gamma \) being replaced by \( 2\gamma \).

\begin{lemma} \label{lemmanormB}
Let $\sigma_\gamma$ be the smallest nonzero singular value of $R_\gamma$, then it holds that  \( \sigma_\gamma \leq 1 \), and
\begin{equation*}
\|\tilde{B}_\gamma\|= \sqrt{1- \sigma_\gamma^2}.
\end{equation*}
\end{lemma}
  
\begin{proof}
From \eqref{notation}, we can see that
  \begin{equation*}
    B_\gamma^* B_\gamma = ( I-Q_\gamma)^* (I-Q_\gamma) = I -  A^* (M_\gamma + M_\gamma^*) A   +Q_\gamma^* Q_\gamma.
  \end{equation*}
By substituting the expression of $Q_\gamma$ and performing a straightforward calculation, we obtain: 
  \begin{align*}
    Q_\gamma^* Q_\gamma &= A^* M_\gamma^* AA^* M_\gamma A
                          = A^*
                          M_\gamma^*  (D_\gamma + L + L^* - \gamma I) M_\gamma A \\
                        &= A^*
                          M_\gamma^*  [(D_\gamma + L) + (D_\gamma + L^*) - D_{2\gamma}]
                          M_\gamma A\\
                        &= A^* M_\gamma^*  \left[ M_\gamma^{-1} + (M_\gamma^{*})^{-1} - D_{2
                          \gamma} \right] M_\gamma A \\
                        &= A^*  (M_\gamma + M_\gamma^*) A - A^*
                          M_\gamma^* D_{2 \gamma} M_\gamma A.
 \end{align*} 
Combining the above equations with the expression in \eqref{notationBR} of $R_\gamma$, we obtain: 
 \begin{align*}
 B_\gamma^* B_\gamma = I - A ^* M_\gamma^* D_{2\gamma} M_\gamma A = I - R_\gamma^* R_\gamma,
\end{align*}
which then implies that \( \sigma_\gamma \leq 1 \).

A simple calculation shows that
\[ 
  \tilde{B}_{\gamma}^{\ast}  \tilde{B}_{\gamma} = P (I - R_{\gamma}^{\ast} R_{\gamma}) P = P -
   PR_{\gamma}^{\ast} R_{\gamma} P = P - R_{\gamma}^{\ast} R_{\gamma} P = (I -
   R_{\gamma}^{\ast} R_{\gamma}) P.
\]
Let \(  U \Sigma V^{\ast} \) be the compact singular value decomposition (SVD) of \( R_\gamma \). Then we have $P = VV^{\ast}$ and
\[ \tilde{B}_{\gamma}^{\ast}  \tilde{B}_{\gamma} = (I - V \Sigma^2 V^{\ast})
   VV^{\ast} = VV^{\ast} - V \Sigma^2 V^{\ast} = V (I - \Sigma^2) V^{\ast}. 
\]
Letting $\lambda_{\text{max}}$ denote the largest singular value of a square matrix, we get
\[ \| \tilde{B}_{\gamma} \| = \sqrt{\lambda_{\text{max}}
   (\tilde{B}_{\gamma}^{\ast}  \tilde{B}_{\gamma})} = \sqrt{1 - \sigma_\gamma^2}, 
\]
which completes the proof. 
\end{proof}

The following result is concerned with the accumulated data error  \( \tilde{q}_\gamma^k-q_\gamma^k \) of the regularized Kaczmarz method. 

\begin{theorem}\label{thm:ek}
  Let \( e^k = \tilde{q}_\gamma^k-q_\gamma^k \), then

  \begin{equation} \label{eq:4}
    \|e^k\| \leq \left[ \frac{1 - (1 - \sigma_{\gamma}^2)^{\frac{k}{2}}}{1 - (1 -
        \sigma_{\gamma}^2)^{\frac{1}{2}}} \right] \left\|A^* M_\gamma(\tilde{p}-p)\right\|.
  \end{equation}
\end{theorem}

\begin{proof}
Using \eqref{equivKM}, we can write 
\begin{equation*}
e^{k+1} = B_\gamma e^k + A^* M_\gamma (\tilde{p} - p),
\end{equation*}
which leads to 
\begin{equation}\label{eq:ek1}
e^k = \left[ \sum_{i=0}^{k-1} (B_\gamma)^i \right] A^* M_\gamma (\tilde{p} - p).
\end{equation}
Since $\mathcal{R}(A^*)$ is an invariant subspace of $B_\gamma$, we can further deduce from \eqref{eq:ek1} that 
\begin{equation} \label{eq:2}
e^k = \left[  \sum_{i=0}^{k-1} (\tilde{B}_\gamma)^i\right] A^* M_\gamma (\tilde{p} - p).
\end{equation}
Let $S_\gamma^k = \sum_{i=0}^{k-1} (\tilde{B}_\gamma)^i$. Utilizing Lemma  \ref{lemmanormB}, we have
\begin{equation*} \label{eq:3}
  \| S_{\gamma}^k \| \leqslant \sum_{i = 0}^{k - 1} \| \tilde{B}_{\gamma} \|^i
= \frac{1 - \| \tilde{B}_{\gamma} \|^k}{1 - \| \tilde{B}_{\gamma} \|} =
\frac{1 - (1 - \sigma_{\gamma}^2)^{\frac{k}{2}}}{1 - (1 -
\sigma_{\gamma}^2)^{\frac{1}{2}}} .
\end{equation*}
Combining the above inequality with \eqref{eq:2},  we obtain the estimate given in \eqref{eq:4}.
\end{proof}

The following corollary further provides an upper bound of the accumulated data error. 

\begin{corollary}\label{cor:k}
The error \(e^k =  \tilde{q}_\gamma^k-q_\gamma^k\) in the regularized block Kaczmarz algorithm satisfies the following estimates: 
\begin{itemize}
 \item[(i)] if $2\leq k\leq2\sigma_\gamma^{-2}$, then 
\begin{equation}\label{eq:estimate_error1}
\Vert e^k\Vert \leq\frac{  \delta_\gamma \sigma_\gamma^2 }{2(1 - (1 -
        \sigma_\gamma^2)^{\frac{1}{2}})} k;
        \end{equation}
\item[(ii)] if $k>2\sigma_\gamma^{-2}$, then 
\begin{equation}\label{eq:estimate_error2}
\Vert e^k\Vert \leq \frac{\delta_\gamma }{1 - (1 -
        \sigma_{\gamma}^2)^{\frac{1}{2}}},
\end{equation}
\end{itemize}
where \(\delta_\gamma=\left\|A^* M_\gamma(\tilde{p}-p)\right\|\).
\end{corollary}

\begin{proof}
We start by introducing the function 
  $$\Psi( \sigma, k)=\frac{1 - (1 - \sigma^2)^{\frac{k}{2}}}{1 - (1 -
        \sigma^2)^{\frac{1}{2}}}.$$
         It follows from Lemma \ref{lemmanormB} that  $0< \sigma_\gamma\leq 1$. It is easy to verify that the function $\Psi( \sigma_\gamma, k)$ admits an upper bound for any $k$:
        \begin{equation}\label{est:psi_cons}
        \Psi(\sigma_\gamma,k)\leq \frac{1}{1 - (1 -
        \sigma_{\gamma}^2)^{\frac{1}{2}}}.
        \end{equation}
        Substituting \eqref{est:psi_cons}  into \eqref{eq:4} leads to the upper bound \eqref{eq:estimate_error2}  for any $k$. 
        
      For $k\geq 2$,  we observe that the function $1- (1-\sigma_\gamma^2)^{\frac{k}{2}} $ satisfies the estimate 
\begin{equation*}
1- (1-\sigma_\gamma^2)^{\frac{k}{2}} \leq \frac{k}{2}\sigma_\gamma^2.
 \end{equation*}
This implies that
\begin{equation}\label{est:psi1}
\Psi( \sigma_\gamma, k)\leq \frac{ {k}\sigma_\gamma^2 }{2(1- (1 -
        \sigma_\gamma^2)^{\frac{1}{2}})}.
 \end{equation}
 When $k\leq 2\sigma_\gamma^{-2}$, the estimate in \eqref{est:psi1} can be further reduced to
$$\frac{ {k}\sigma_\gamma^2 }{2(1 - (1 -
        \sigma_\gamma^2)^{\frac{1}{2}})}\leq  \frac{1}{1 - (1 -
        \sigma_{\gamma}^2)^{\frac{1}{2}}}.$$
       For $2\leq k\leq 2\sigma_\gamma^{-2}$, substituting \eqref{est:psi1} into \eqref{eq:4} yields a sharper upper bound \eqref{eq:estimate_error1}.
\end{proof}

\begin{remark}
 The above corollary indicates that the data error is bounded above by an increasing function of $k$ when $k$ is smaller than $2\sigma_\gamma^{-2}$, which is consistent with the observed accumulation of data errors during iteration in numerical experiments. In practice, the value of $\sigma_\gamma^{-2}$ might be very large. Although it is unclear how to achieve a sharper estimate in Corollary \ref{cor:k}, the result tracks the behavior of the data error.
\end{remark}

In the following analysis, we examine the impact of data perturbations on the reconstruction process, taking the example of reconstructing the mean $g$. For simplicity, we use the same notation to represent the function and its discrete values obtained from measurements or sampling. Recall that the data $p_j$ corresponds to measurements of $\mathbb{E}(\Re u(x,\kappa_j))$. Due to data noise, the perturbed data $\tilde{p}_j$ is modeled by $\mathbb{E}_N(\Re \tilde{u}(x,\kappa_j))$, where $N$ denotes the number of realizations, $\mathbb{E}_N$ stands for the sample mean, and $\tilde{u}(x,\kappa_j)$ represents the measurement of $u(x,\kappa_j)$ with additive noise. We use $\tilde{u}_j$ and $u_j$ to denote $\tilde{u}(x,\kappa_j)$ and $\mathbb{E}(u(x,\kappa_j))$, respectively. It should be noted that the analysis for reconstructing $h^2$ follows a similar process.

By \eqref{eq:mildsolution}, the wave field $u_j$ satisfies 
\begin{equation*}
u(x,\kappa_j) = \int_D G(x,y,\kappa_j) g(y) dy +\int_D G(x,y,\kappa_j) h(y) dW_y,\quad a.s..
\end{equation*}
Neglecting the discretization error, we can express $u_j$ and $\tilde{u}_j$ as
\begin{equation*} 
\Re u_j = A_j g,\quad \Re \tilde{u}_j =  A_j (g+h \dot{W}_x) + r_j, 
\end{equation*}
where $h \sim \mathcal{N} (0, H)$, $r_j \sim \mathcal{N} (0, \delta^2 I)$ with $H = \text{diag}
[h (x_i)^2]$ and $\delta$ indicating the noise level. We define 
\[ u_R = \left[\begin{array}{c}
     \Re u_1\\
     \vdots\\
     \Re u_m
   \end{array}\right], \quad \tilde{u}_R = \left[\begin{array}{c}
     \Re \tilde{u}_1\\
     \vdots\\
     \Re \tilde{u}_m
   \end{array}\right], \quad \text{and} \quad r = \left[\begin{array}{c}
     r_1\\
     \vdots\\
     r_m
   \end{array}\right] . \]
We further assume that the components of $h$, $r_1, r_2, \ldots, r_m$ are independent, leading to 
\[ \tilde{u}_R \sim \mathcal{N} (Ag, AHA^{\ast} + \delta^2 I) . \]
Thus, we obtain 
 $$ \tilde{u}_R- u_R \sim  \mathcal{N} (0, \Theta), $$
where  $\Theta  = AHA^{\ast} + \delta^2 I$. 
Moreover, from the definition of $\tilde{p}$ and the central limit theorem, we deduce that 
$$  \sqrt{N}(\tilde{p} -p ) \sim  \mathcal{N} (0, \Theta). $$

The following corollary provides an estimate for the accumulated data error in the distribution sense, and is a direct consequence of Theorem \ref{thm:ek}.

\begin{corollary}\label{cor:distri}
  Let \( e^k = \tilde{q}_\gamma^k-q_\gamma^k \). Then
  \begin{equation*}
    \sqrt{N} e^k \xrightarrow{d} \mathcal{N} (0, \Lambda_k), 
  \end{equation*}
  where
  \begin{equation*}
    \|\Lambda_k\| \leq \left[ \frac{1 - (1 - \sigma_{\gamma}^2)^{\frac{k}{2}}}{1 - (1 -
        \sigma_{\gamma}^2)^{\frac{1}{2}}} \right]^2 \left\|A^* M_\gamma\right\|^2\|\Theta\|.
  \end{equation*}
\end{corollary}

These results demonstrate the dependence of the perturbation in the data and the variance of the accumulated data error on the number of realizations $N$ and the noise level $\delta$.

\section{Stage two: machine learning techniques}\label{Sec:4}

In this section, we present four machine learning techniques as the second stage of our proposed method. Let $\{(X_i, Y_i)\}_{i=1}^M$ be the training dataset, where $X_i\in\mathbb{R}^{n\times n}$ represents a matrix computed using the measurement data corresponding to the $i$-th sample from the first stage, $Y_i\in \mathbb{R}^{n\times n}$ denotes the exact mean or variance of the $i$-th sample after discretization, and $M$ is the total number of samples. The goal of the second stage is to learn the mapping from the approximations obtained in the first stage to the exact statistical properties by utilizing the training dataset. We present two methods based on linear regression and two methods based on neural networks.

\subsection{Principal component analysis}

Principal component analysis (PCA) is a well-established approach for dimensionality reduction in statistical analysis \cite{hotelling1933analysis}. Recently, Bhattacharya et al. \cite{bhattacharya2021model} combined PCA with neural network techniques to solve inverse problems of parametric partial differential equations. In our work, we build upon the idea of PCA to reduce the dimension of input and output spaces, and then apply linear regression as the second stage of our algorithm.

To prepare for training, we flatten each $X_i$ into a vector $\mathbf{x}_i\in \mathbb{R}^{n^2}$ and construct the input data matrix $\mathbf{X}\in  \mathbb{R}^{n^2\times M}$ by concatenating $\mathbf{x}_i$ as columns. Similarly, we construct the output data matrix $\mathbf{Y}\in  \mathbb{R}^{n^2\times M}$ with the flattened $Y_i$, denoted by $\mathbf{y}_i\in \mathbb{R}^{n^2}$. After concatenating the input and output data matrices, we perform mean subtraction and use $\mathbf{X}$ and $\mathbf{Y}$ to denote the data matrices after the mean centering step for convenience. Then, we utilize singular value decomposition (SVD) to derive the low-dimensional representation of the data. Specifically, we first implement singular value decomposition  
$$\mathbf{E} : = \begin{bmatrix} \mathbf{X}\\ \mathbf{Y}\end{bmatrix} = \mathbf{U}\mathbf{\Sigma}  \mathbf{V}^*,$$
where $\mathbf{U}\in \mathbb{R}^{2n^2\times 2n^2}$ and $\mathbf{V}\in\mathbb{R}^{M\times M}$ are orthonormal matrices, and $\mathbf{\Sigma}\in  \mathbb{R}^{2n^2\times M}$ is the diagonal singular value matrix.
Then, we keep the first $\sP$ principal components to form a low-rank approximation of the data matrix $\mathbf{E}$. Let $\mathbf{\Psi}_\sP = [\mathbf{u}_1 \ \mathbf{u}_2\ ...\ \mathbf{u}_\sP]$, where $\mathbf{u}_i$ denotes the $i$-th column of $\mathbf{U}$. From the Eckart--Young theorem, we conclude that $\mathbf{\Psi}_\sP$ minimizes
 $$\Vert \mathbf{E}-\mathbf{\Psi} \mathbf{\Psi}^* \mathbf{E}\Vert_F $$
over all orthonormal $\mathbf{\Psi}\in \mathbb{R}^{2n^2\times \sP}$, where $\Vert \cdot\Vert_F$ is the Frobenius norm.

To predict reconstructions in the testing process, we use linear regression with the selected principal components. Denote 
\begin{equation*}
\mathbf{\Psi}_\sP=   \begin{bmatrix} \mathbf{\Psi}_X\\ \mathbf{\Psi}_Y\end{bmatrix},
\end{equation*}
where $\mathbf{\Psi}_X, \mathbf{\Psi}_Y\in \mathbb{R}^{n^2\times \sP}$.
Given an matrix $X_{\text{test}}$ obtained from the first stage,  we flatten it into a vector $\mathbf{x}_{\text{test}} \in \mathbb{R}^{n^2}$, and compute the corresponding output $\mathbf{y}_{\text{test}}$ as
\begin{equation*}
 \mathbf{y}_{\text{test}} = \mathbf{\Psi}_Y\mathbf{\Psi}_X^\dagger \mathbf{x}_{\text{test}}.
\end{equation*}
where $\mathbf{\Psi}_X^\dagger$ denotes the pseudo-inverse of $\mathbf{\Psi}_X$. We finally obtain the prediction $Y_{\text{test}}$ by reshaping $\mathbf{y}_{\text{test}}$ into an $n \times n$ matrix.

\subsection{Dynamic mode decomposition}

The dynamic mode decomposition (DMD) is a powerful technique for dimensionality reduction of dynamic systems \cite{schmid2010dynamic}. By casting the mapping between the approximation and the exact statistical properties as a dynamical system, DMD can be leveraged to linearly approximate the inverse mapping for our problem.

We assume that the columns of data matrices $\mathbf{X}$ and $\mathbf{Y}$ obey a dynamical system of the form $\mathbf{y}_i= F(\mathbf{x}_i)$. The objective of DMD is to approximate this dynamic system using a linear operator $\mathbf{Y} = \mathbf{A}\mathbf{X}$. To this end, we employ the truncated SVD for matrix $\mathbf{X}$ to obtain a rank-$\sD$ low-dimensional approximation of the form $\mathbf{X}_\sD = \mathbf{U}_\sD \mathbf{\Sigma}_\sD \mathbf{V}_\sD^*$, where $\mathbf{U}_\sD\in \mathbb{R}^{n^2\times \sD}$, $\mathbf{V}_\sD\in\mathbb{R}^{M\times \sD}$, and $\mathbf{\Sigma}_\sD\in \mathbb{R}^{\sD \times \sD}$. The optimal linear transformation $\mathbf{A}_\sD$ is then computed from the reduced dimension-$\sD$ coefficient space as
$$ \mathbf{A}_\sD=  \mathbf{U}_\sD^*\mathbf{Y} \mathbf{V}_\sD \mathbf{\Sigma}_\sD^{-1}.    $$
Computing the eigendecomposition of matrix \(\mathbf{A}_\sD\) leads to an approximation of the spectrum of the underlying Koopman operator of the system.

To predict outputs using the linearized system, we take a matrix $X_{\text{test}}$ generated in the first stage and flatten it to obtain $\mathbf{x}_{\text{test}}$. Subsequently, the flattened prediction $\mathbf{y}_{\text{test}}$ is computed via
\begin{equation*}
\mathbf{y}_{\text{test}} =\mathbf{U}_\sD \mathbf{A}_\sD  \mathbf{U}_\sD^*\mathbf{x}_{\text{test}} =  \mathbf{Y} \mathbf{V}_\sD \mathbf{\Sigma}_\sD^{-1}  \mathbf{U}_\sD^*\mathbf{x}_{\text{test}}.
\end{equation*}
Finally, we obtain the prediction $Y_{\text{test}}$ by reshaping $\mathbf{y}_{\text{test}}$ into an $n \times n$ matrix.

\subsection{U-Net based network}

The disparity between the reconstructed profile from stage one, exhibiting a smooth image, and the true sharp profile of the source (as shown in Figure \ref{fig:compare}), presents an opportunity to leverage image segmentation tools for learning the mapping between them. In this regard, we test the applicability of the U-Net convolutional neural network \cite{ronneberger2015u}, which offers a unique encoder-decoder architecture that delivers highly accurate image segmentation. As shown in Figure \ref{Unet},  the U-Net approach preserves features of the input images with skip connections. The architecture of the neural network consists of a contracting path and an expansive path.  The contracting path involves sequentially applying a convolution, followed by a batch normalization layer and a ReLU activation function, to perform downsampling. The expansive path comprises of three components: an upsampling of the feature map, a batch normalization layer, and a ReLU activation function. The first component is followed by an up-convolution that halves the number of feature channels  and the third component is used to restore the image to the size of the input image. For each convolutional layer in the contracting path, the feature maps are passed to the corresponding layers in the expansive path via the skip connections.
 
 In the training process of the U-Net model $\chi$, we adopt the $L^1$ error as the loss function, i.e., we minimize the following functional over the parameter $\theta$ of the neural network model $\chi$:
\begin{equation*}
\mathcal{L}_{L^1}(\chi)= \mathbb{E} (\Vert Y-\chi(X;\theta)\Vert_{L^1}),
\end{equation*}
where $\mathbb{E}$ denotes the expectation of the underlying distribution from which the training dataset is sampled. In practice, the expectation is approximated by taking the average value of the loss function over a batch of training examples. Empirical results suggest that this choice of $L^1$ error yields more favorable performance as compared to the mean squared error.

\begin{remark}
If we consider the first stage of the overall method as a layer of the neural network architecture, it may be regarded as a deconvolution layer that incorporates a priori understanding of the physical characteristics.
\end{remark}

\begin{figure}
\centering
\includegraphics[width=0.9\textwidth]{ 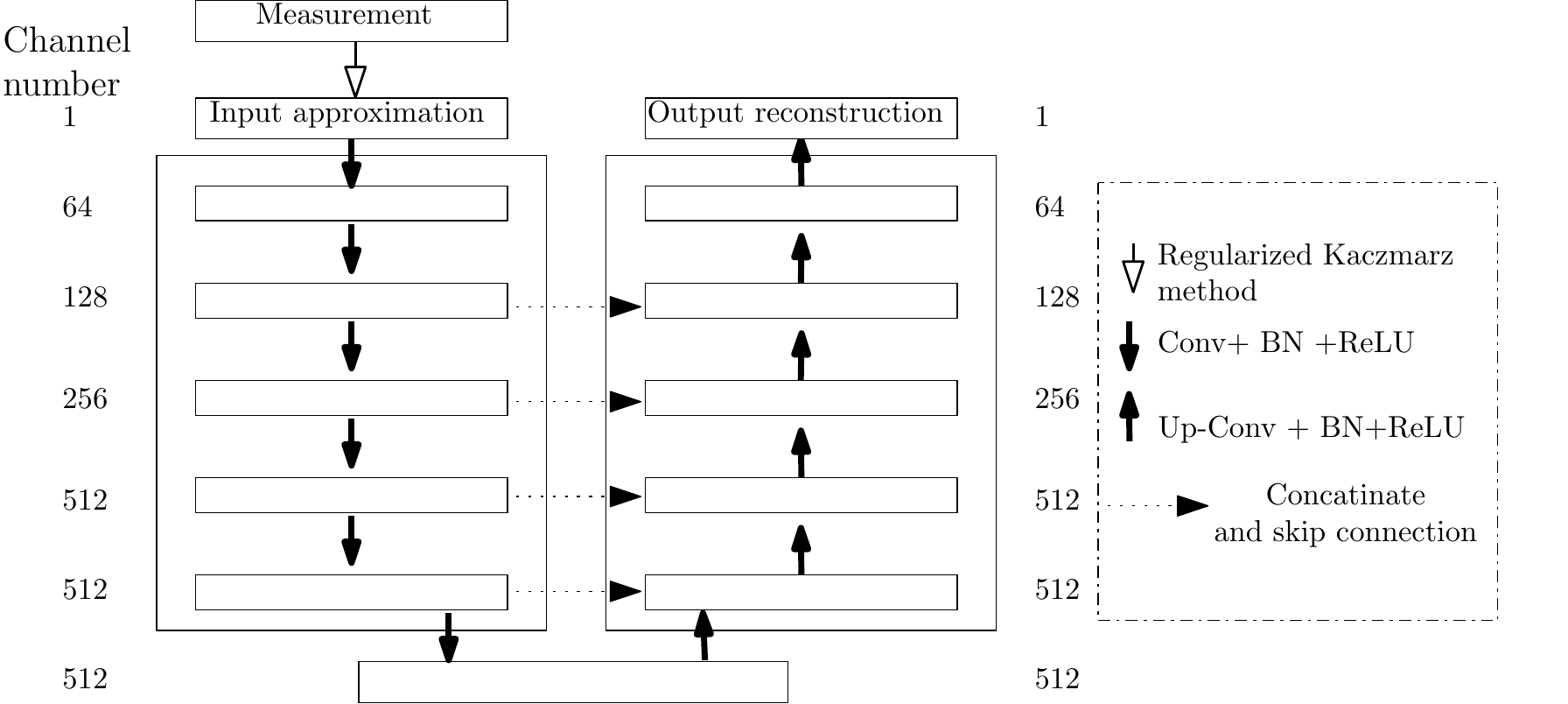}
\caption{The U-Net architecture.}\label{Unet}
\end{figure}

\subsection{Pix2pix neural network}

As an alternative to the U-Net model, we also evaluated the pix2pix neural network architecture. The pix2pix structure belongs to the family of conditional generative adversarial networks (cGANs), comprising a discriminator network $\D$ and a generator network $\G$, which cooperate in order to generate realistic reconstructions \cite{isola2017image}. 

In our experiments, the generator network maps an input image, $X_i\in \mathbb{R}^{n\times n}$, derived from the first stage, to an associated reconstruction, $\G(X_i) \in \mathbb{R}^{n\times n}$. On the other hand, the discriminator network takes in a pair of images, $(X, Y) \in \mathbb{R}^{n\times n} \times  \mathbb{R}^{n\times n}$, as input and produces a scalar-valued output ranging between 0 and 1. The two networks are jointly trained in a competitive manner, wherein the generator network endeavors to generate images $\G(X_i)$ that resemble the target images $Y_i$, while the discriminator network attempts to differentiate between real image pairs $(X_i, Y_i)$ and generated image pairs $(X_i, \G(X_i))$, with the output indicating the probability that an input image pair is real (i.e., it originates from the training dataset) or generated (i.e., it was produced by the generator network). The training process is illustrated in Figure \ref{pix2pix_diagram}.

The pix2pix algorithm utilizes a loss function consisting of two distinct components. The first component, denoted as the adversarial loss, is expressed in terms of a binary cross-entropy loss, and is represented as follows: 
\begin{equation*}
 \mathcal{L}_{\mathrm{cGAN}}(\G,\D) = \mathbb{E} (\log \D(X,Y)) + \mathbb{E}(\log(1-\D(X,\G(X)))).
 \end{equation*}
The second component of the loss function is referred to as the reconstruction loss, and is defined as follows: 
\begin{equation*}
\mathcal{L}_{L^1}(\G) = \mathbb{E} (\Vert Y-\G(X)\Vert_{L^1}).
\end{equation*}
The total loss is then given by
\begin{equation*}
 \mathcal{L}(\G,\D) = \mathcal{L}_{\mathrm{cGAN}}(\G,\D) + \lambda \mathcal{L}_{L^1} (\G),
\end{equation*}
where $\lambda$ represents a weight term that determines the relative importance of each component in the overall loss function. In practice, the expectations involved in the loss functions are approximated by sample means over batches of the training data. 
In the training process, we  solve the optimization problem
\begin{equation*}
\arg \min_{\theta_\G} \max_{\theta_\D} \mathcal{L}(\G, \D),
\end{equation*}
where $\theta_\G$ and $\theta_\D$ denote the parameters of the neural network models $\G$ and $\D$, respectively.

Following the approach proposed in \cite{isola2017image}, we employ the U-Net and patch-GAN architectures for the generator and discriminator networks respectively. This choice enables improved performance and efficiency in a variety of image-to-image translation tasks. 


\begin{figure}
\centering
\includegraphics[width=0.95\textwidth]{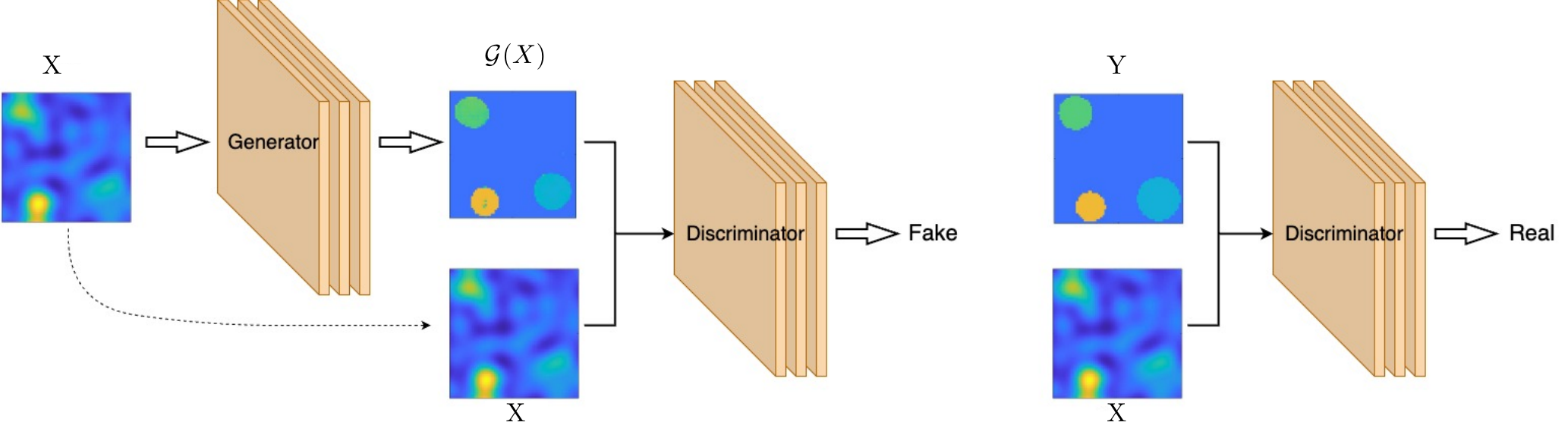}
\caption{The pix2pix algorithm. The generator $\G$ 
learns to fool the discriminator $\D$,  while the discriminator $\D$ learns to distinguish the fake $(X, \G(X))$ and real $(X,Y)$ tuples in the training set. }\label{pix2pix_diagram}
\end{figure}

\section{Numerical experiments}\label{Sec:5}

In this section, we present a comparative study of the data-assisted methods for the inverse random source problem.

\subsection{Preparation for training}

For the reconstruction of the mean function \( g \), the training dataset is generated as follows. The support of \( g \) is the union of three disks \( D_l \) with radius \( r_l \) and center \( (a_l, b_l) \), \( l=1,2,3 \). For any fixed \( l \), the radius \( r_l \) is generated from the uniform distribution on \( [0.2, 0.4] \). The coordinates \( a_l \) and \( b_l \) are generated from the uniform distribution in \( [-1+r_l, 1-r_l] \) so that the disk is randomly located but always contained in the rectangle \( D = [-1, 1] \times [-1, 1] \). Note that the disks are allowed to overlap. All the random numbers are generated independently from each other, with different but fixed seeds so that the results can be reproduced. The source term \( g \) is set to \( 0 \) in \( D \), and \( l \) in \( D_l \) for \( l=1,2,3 \) in a successive manner. The source term \( h \) is fixed as
\begin{equation}\label{eq:h}
  h(x) = 0.6 e^{-8 \left[ (x_1^2+x_2^2)^{1.5} - 0.75 (x_1^2+x_2^2) \right]}.
\end{equation}
The pseudocolor plot of the function \( h(x_1, x_2) \) is shown in Figure \ref{fig:parameters} (A).
\begin{figure}[ht!]
  \centering
  \begin{subfigure}[b]{0.3\textwidth}
    \centering
    \includegraphics[width=\textwidth]{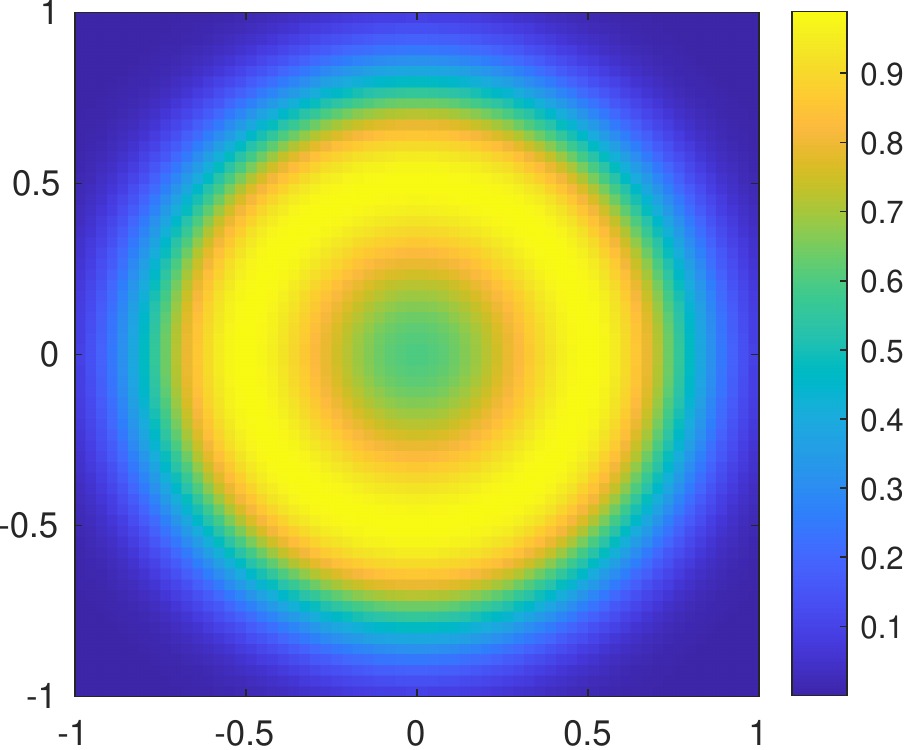}
    \caption{}
  \end{subfigure}
  \hfill
  \begin{subfigure}[b]{0.3\textwidth}
    \centering
    \includegraphics[width=\textwidth]{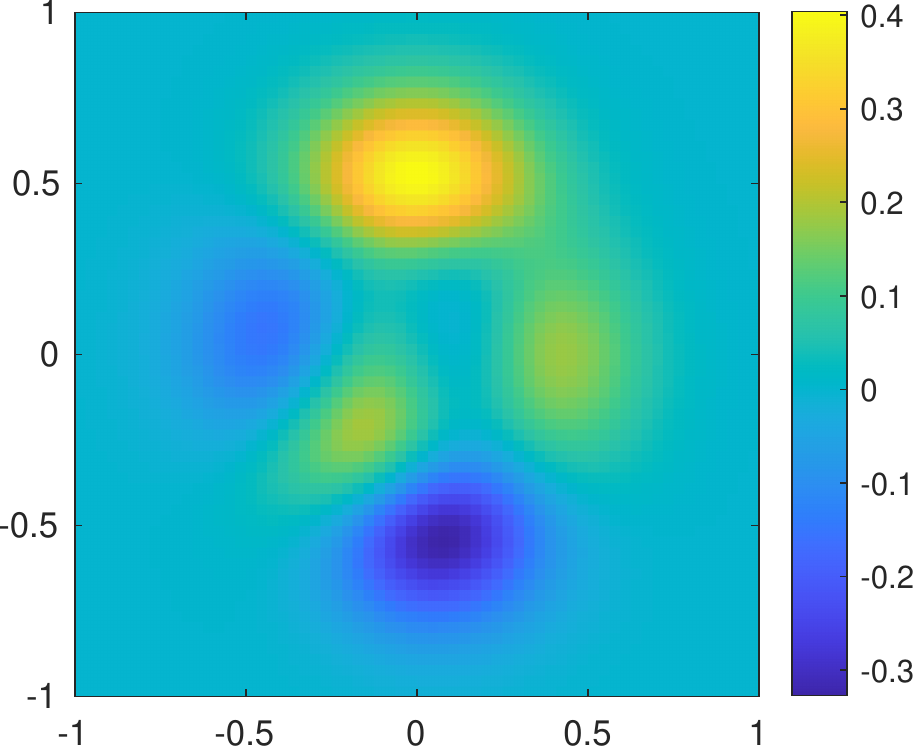}
    \caption{}
  \end{subfigure}
  \hfill
  \begin{subfigure}[b]{0.3\textwidth}
    \centering
    \includegraphics[width=\textwidth]{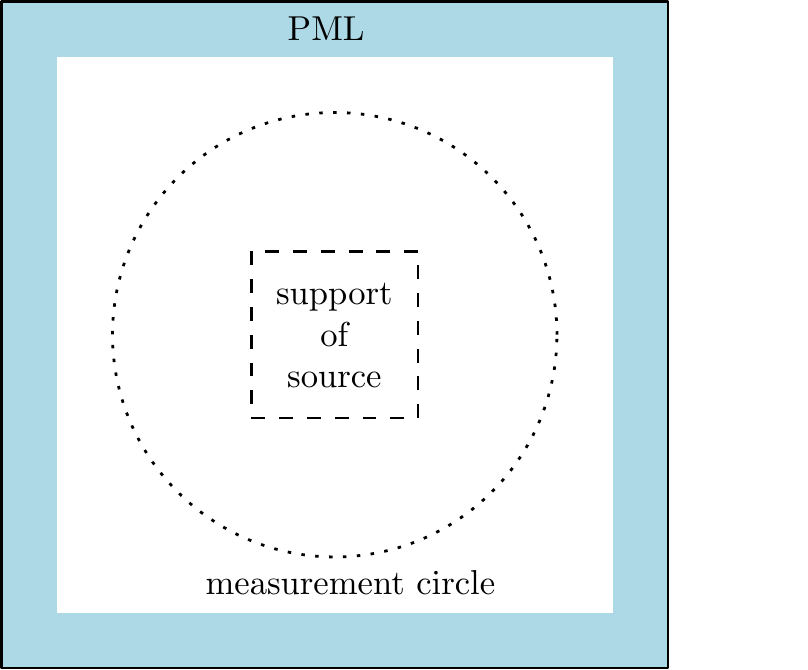}
    \caption{}
  \end{subfigure}
  \caption{(A) Pseudocolor plot of the function \( h \); (B) Pseudocolor plot of the function \( \eta \); (C) Computational domain for the direct problems.}
  \label{fig:parameters}
\end{figure}

To generate the training dataset for the reconstruction of the variance function $h$, we switch the roles of $g$ and $h$. More specifically, we consider the mean function $g$ to be fixed and given by \eqref{eq:h}. In this setting, we set $h$ to be equal to $0$ within the domain $D$, and assign the values of $1+\frac{l}{2}$ inside each of the three respective disks, $D_l$, in a stepwise manner for $l=1,2,3$. The disks $D_l$ are generated in accordance with the same method as that used for constructing the mean function $g$.

For experiments with an inhomogeneous medium, the function $\eta(x)$ is fixed as \( \eta(x) = q(3x_1, 3x_2) \), where
\[
  q(x) = 0.5 \left[ 0.3 (1 - x_1)^2 e^{ -x_1^2-(x_2+1)^2} - 
(0.2 x_1-x_1^3-x_2^5) e^{-x_1^2-x_2^2} - 0.03 e^{-(x_1+1)^2-x_2^2} \right].
\]
The pseudocolor plot of \( \eta(x) \) is shown in Figure \ref{fig:parameters} (B).

The infinite domain is truncated to a finite one using the technique of perfectly matched layer (PML). The computational domain is \( [-3, 3] \times [-3, 3] \) with a rectangular PML layer of thickness \( 0.5 \). The direct problem is solved by the finite element method (FEM) and the synthetic data is measured at \( 32 \) uniformly distributed points on the circle centered at the origin with radius \( 2 \). For any fixed source terms, the direct problem is solved for \( 5 \) wavenumbers \(\kappa_j = (0.5+j)\pi, j=0,1,2,3,4\). A schematic of the computational domain for the direct problems is shown in Figure \ref{fig:parameters} (C).

To generate random numbers, the GNU Scientific Library \cite{Galassi2002GNU} is used. Mesh generation is carried out using Gmsh \cite{Geuzaine2009Gmsh}, and the finite element method (FEM) is implemented by FreeFEM \cite{Hecht2012New}. To accelerate the computation of numerous direct problems, PETSc \cite{Balay1998PETSc} is employed to solve linear systems with multiple right-hand sides for different realizations. Moreover, Open MPI is utilized for parallel computing with varying source terms and wavenumbers.

\subsection{Stage one configuration}

The direct problem is solved for fixed source terms \( g \) and \( h \), with \( 5 \) distinct wavenumbers \(\kappa_j = (0.5+j)\pi, j=0,1,2,3,4\). The program is run for each fixed \(\kappa_j\) using either \(100\) realizations, for the reconstruction of the mean, or \(1000\) realizations, for the reconstruction of the variance. The resulting data on the measurement circle is processed using the Kaczmarz algorithm for a homogeneous \cite{bao2016inverse} medium to produce initial reconstructed images. The reconstructed images are sampled at a uniform grid size of \( 64 \times 64 \) within the range of \([-1, 1] \times [-1, 1]\). For regularization, the parameter \( \gamma \) is set to \( 10^{-8} \). The number of outer loops is taken to be \( 1 \).

\subsection{Stage two configuration}
For the PCA-based regression approach, the reduced dimension is set to \(\sP = 1000\) while for the DMD approach, the reduced dimension is set to  \(\sD = 100\). We implement the U-Net architecture comprising of convolutional layers of kernel size \( 4 \times 4 \) and batch normalization layers, following \cite{ioffe2015batch}. The weight $\lambda$ in the pix2pix algorithm is set to \(10\). All learning-based approaches are trained and tested using the PyTorch framework. The training process is accelerated using an Nvidia Tesla P100 GPU.

\subsection{Numerical experiments} 

We present numerical experiments in this section to compare the performance of different data-assisted methods. Unless otherwise specified, the testing dataset consists of $200$ randomly generated samples.
\begin{exam}
 In the first example, we assess the performance of the U-Net method in stage two with training datasets of varying sizes.
\end{exam}

The reconstruction quality of the samples in the testing set is measured by the averaged $L^1$ relative errors of all samples. The accuracy and training time of the U-Net method are reported in Table \ref{table1}. It is observed that as the size of the training dataset increases, the training time increases, and the relative error decreases. However, the accuracy does not improve significantly when the number of examples exceeds 1600. Therefore, we present results with a training dataset consisting of $M =1600$ samples in the subsequent examples to compare the performance of different approaches.

Figure \ref{fig.losshistory} displays the loss versus epochs for the U-Net based method, indicating that the loss decays slowly after 100 epochs. In the following examples, we present the neural networks trained with 100 epochs.
\begin{table}[t]
  \centering
  \caption{The training time and $L^1$ relative error for the U-Net based method with different numbers of training samples }\label{table1}
\begin{tabular}{c |c | c| c | c}
\hline
Size of training set &400& 800& 1600 &3200 \\
\hline
\hline
Time &34 & 51  &      92          &  130  \\
\hline
Relative error &0.24 &  0.21 &    0.20 & 0.19\\
\hline
\end{tabular}
\end{table}

\begin{figure}[t]
  \centering
  \includegraphics[width=0.5\textwidth]{ 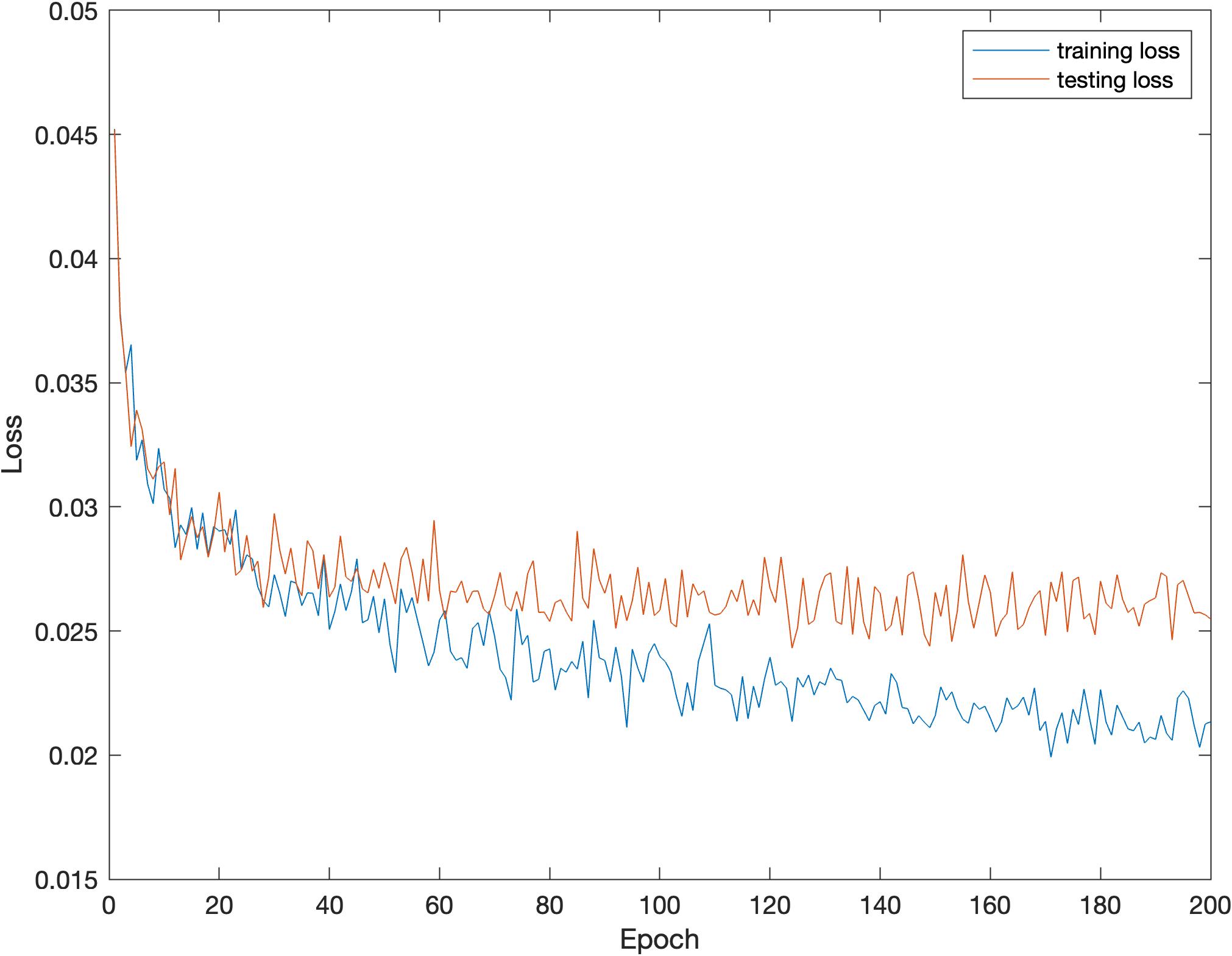} 
\caption{Loss versus epoches for Example 1.} \label{fig.losshistory} 
\end{figure}

\begin{exam}
In this example, we reconstruct the mean function $g$ and variance function $h$ of the random source when the medium is homogeneous.
\end{exam}

We  consider IP1 in this example and present representative pseudocolor plots of reconstructed images in Figures \ref{Ex2.main_mean} and \ref{Ex2.main_var}. Columns 1--2 display the ground-truth images and reconstructions from stage one (the Kaczmarz method), while columns 3--6 show reconstructions after stage two using PCA, DMD, U-Net, and pix2pix approaches.

 The $L^1$ relative error in the testing set is presented in Tables \ref{table2} and \ref{table3}. We observe that the Kaczmarz reconstructions capture the rough locations of the source, but not their boundary and amplitude. In stage two, the PCA and the DMD methods provide better reconstructions, but the background is still not clear and the boundary is blurry. With the two neural network algorithms U-Net and pix2pix, the reconstructions are noticeably enhanced and the homogeneous background is clearly separated from the inclusions. This improvement is also verified by the $L^1$ relative errors in Tables \ref{table2} and \ref{table3}. 
\begin{figure}[t]
\centering
\includegraphics[width=1\textwidth]{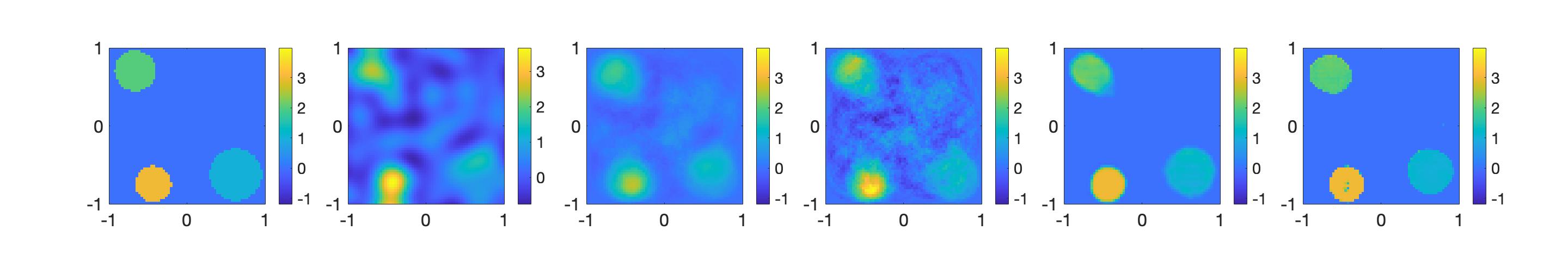} \vskip -1em 
\includegraphics[width=1\textwidth]{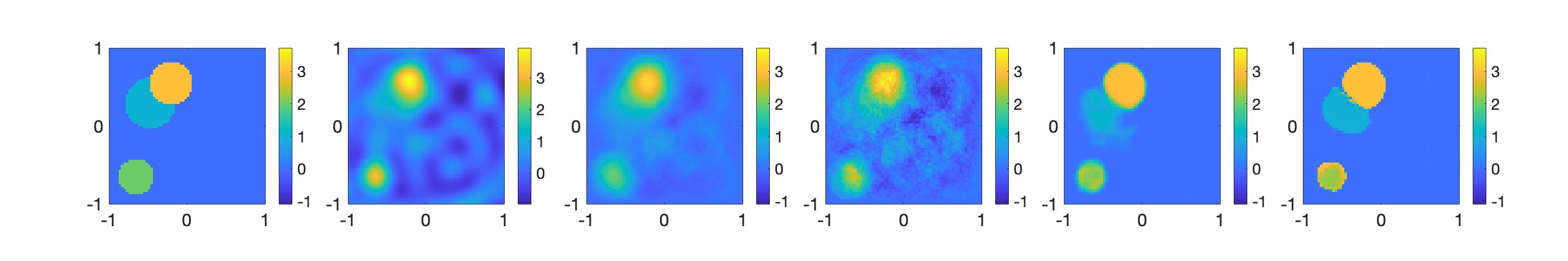} \vskip -1em 
\includegraphics[width=1\textwidth]{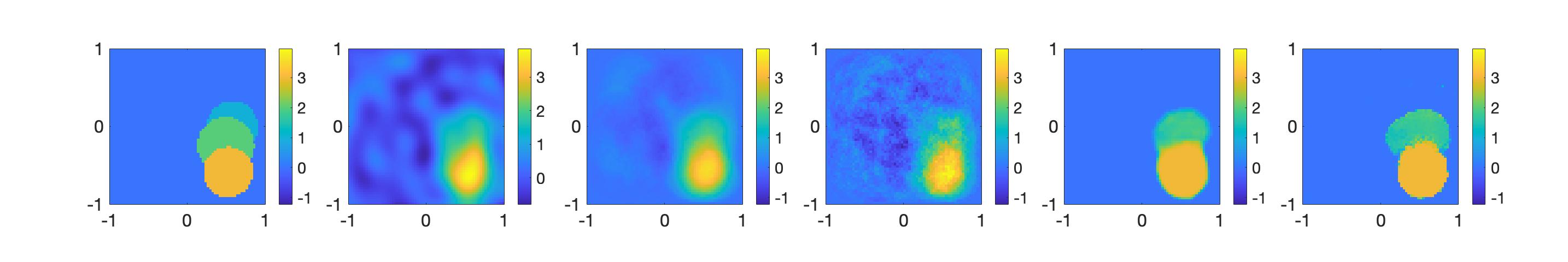} \vskip -1em 
\caption{Pseudocolor plots of the mean functions \( g \) and their reconstructions in a homogeneous medium. The columns from the left are the ground-truth images, the reconstructed images by the Kaczmarz method, PCA-based method, DMD-based method, U-Net based method, and pix2pix method.}
\label{Ex2.main_mean}
\end{figure}

\begin{figure}[t]
\centering
\includegraphics[width=1\textwidth]{ 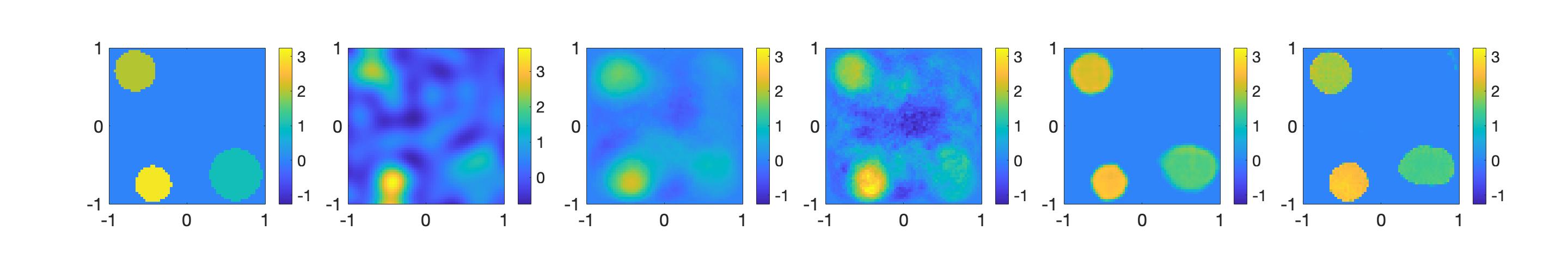} \vskip -1em 
\includegraphics[width=1\textwidth]{ 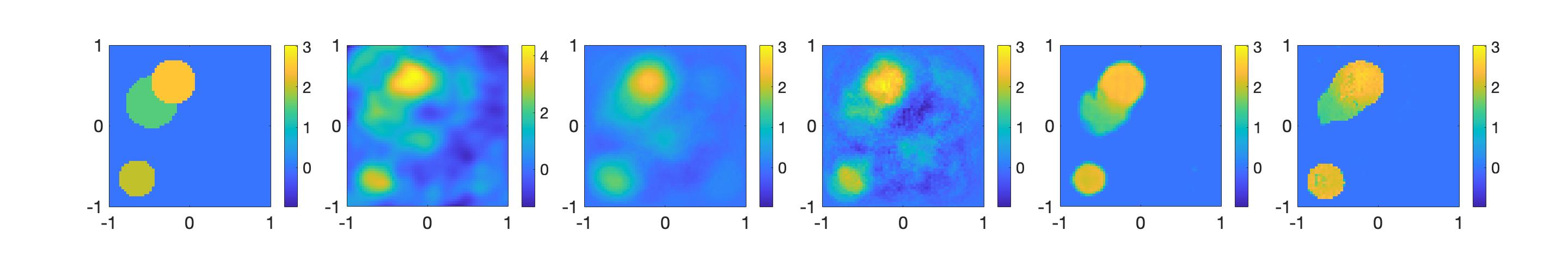} \vskip -1em 
\includegraphics[width=1\textwidth]{ 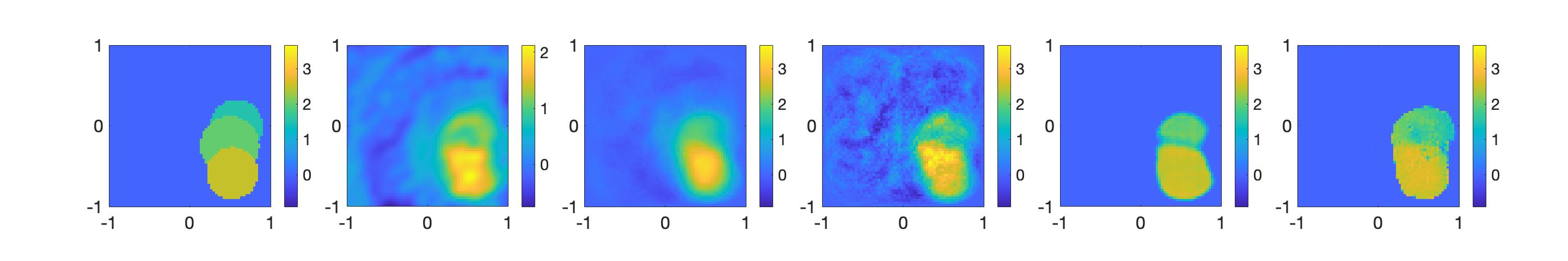} \vskip -1em 
\caption{Pseudocolor plots of the variance functions \( h \) and their reconstructions in a homogeneous medium. The columns from the left are the ground-truth images, the reconstructed images by the Kaczmarz method, PCA-based method, DMD-based method, U-Net based method, and pix2pix method.}
\label{Ex2.main_var}
\end{figure}

\begin{table}[t]
  \centering
  \caption{Training time and $L^1$ relative error for the reconstruction of the mean function \( g \) in a homogeneous medium. }\label{table2}
\begin{tabular}{c | c| c | c| c}

\hline
 Algorithm &  PCA  & DMD & U-Net &pix2pix \\
\hline
\hline
Time &  4 &  6 & 74 & 693 \\
\hline
Relative error & 0.62  &  0.63 & 0.20 & 0.22\\\hline
\end{tabular}
\end{table}

\begin{table}[t]
  \centering
  \caption{Training time and $L^1$ relative error for the reconstruction of the variance function \( h \) in a homogeneous medium. }\label{table3}
\begin{tabular}{c | c| c | c| c}
\hline
 Algorithm &  PCA  &  DMD &  U-Net  &pix2pix \\
\hline
\hline
Time &4  &6  & 74  &  729 \\
\hline
Relative error &0.90  & 0.77 &  0.28&  0.30\\
\hline
\end{tabular}
\end{table}

In the following three examples, the medium is inhomogeneous for both the training and testing datasets. It is worth noting that the refractive index $1 + \eta$ is solely used to solve the direct problem to generate the dataset, but is not available when implementing the two-stage method. Thus, the examples demonstrate the performance of the proposed method for both IP2 and IP3.

\begin{exam}
  In this example, we reconstruct the mean function \( g \) and variance function \( h \) of the random source when the medium is inhomogeneous.
\end{exam}

Figures \ref{Ex3.main_mean} and \ref{Ex3.main_var} present pseudocolor plots of some representative reconstructions. The reconstructions from the first stage are worse than the initial approximations in Example 2, since the Kaczmarz method is derived utilizing the integral equations corresponding to the direct problem with a homogeneous medium. In the subsequent stage, the PCA and DMD approaches provide enhanced reconstructions; however, the background is still unclear. The U-Net and pix2pix algorithms lead to much improved reconstructions with a clean background, sharp boundary, and correct amplitude. The improvements are also confirmed by the $L^1$ relative errors shown in Tables \ref{table4} and \ref{table5}, indicating that the hidden information of the inhomogeneity has been successfully captured by the learning methods.

\begin{figure}[t]
\centering
\includegraphics[width=1\textwidth]{ 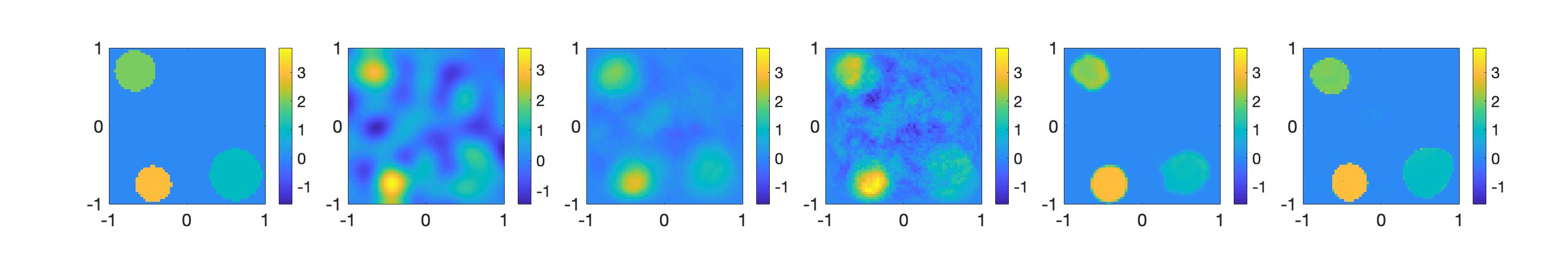} \vskip -1em 
\includegraphics[width=1\textwidth]{ 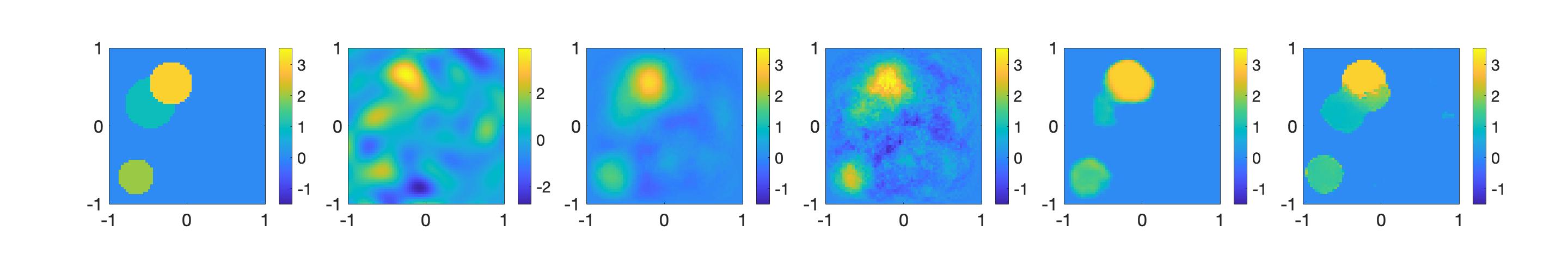} \vskip -1em 
\includegraphics[width=1\textwidth]{ 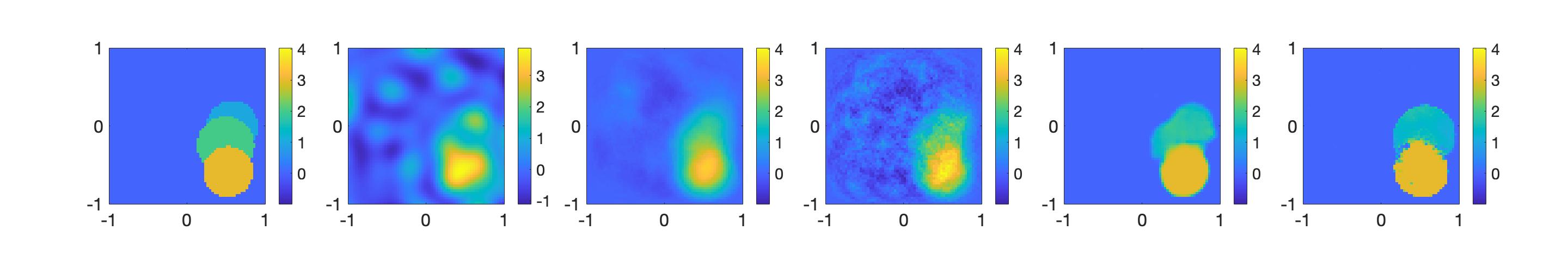} \vskip -1em 
\caption{Pseudocolor plots of the mean functions \( g \) and their reconstructions in an inhomogeneous medium. The columns from the left are the ground-truth images, the reconstructed images by the Kaczmarz method, PCA-based method, DMD-based method, U-Net based method, and pix2pix method.}
\label{Ex3.main_mean}
\end{figure}

\begin{figure}[t]
\centering
\includegraphics[width=1\textwidth]{ 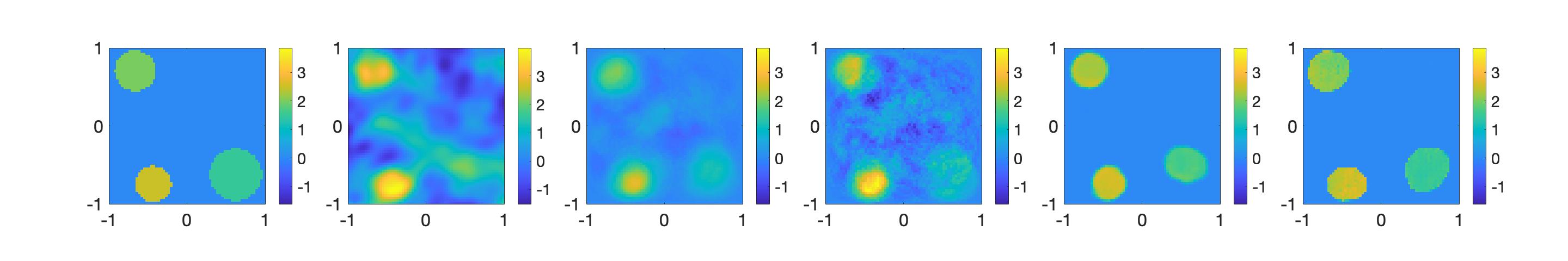} \vskip -1em 
\includegraphics[width=1\textwidth]{ 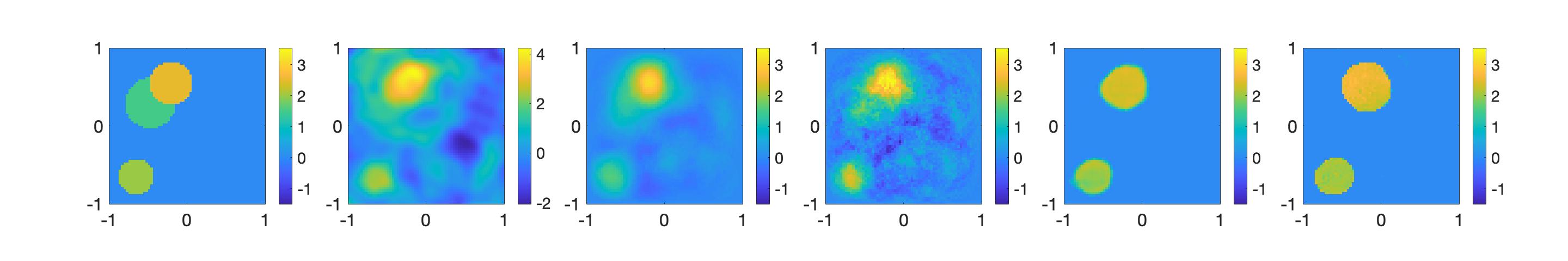} \vskip -1em 
\includegraphics[width=1\textwidth]{ 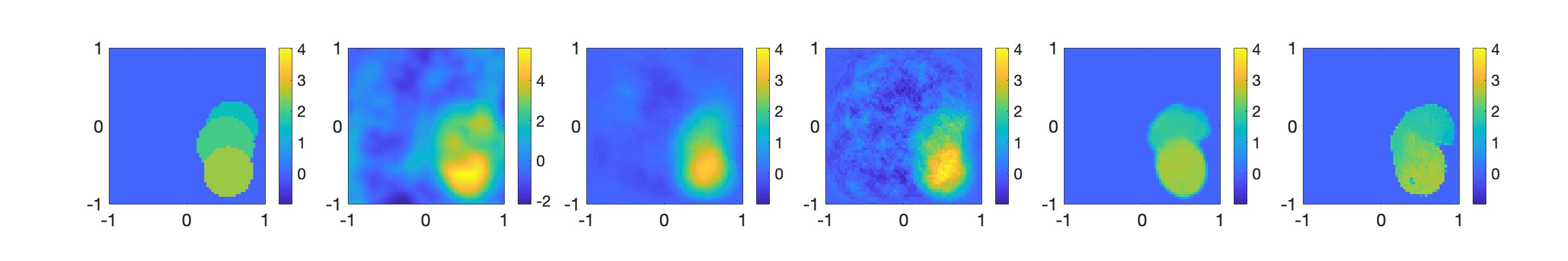} \vskip -1em 
\caption{Pseudocolor plots of the variance functions \( h \) and their reconstructions in an  inhomogeneous medium. The columns from the left are the ground-truth images, the reconstructed images by the Kaczmarz method, PCA-based method, DMD-based method, U-Net based method, and pix2pix method.}
\label{Ex3.main_var}
\end{figure}

\begin{table}[t]
  \centering
  \caption{Training time and $L^1$ relative error for the reconstruction of the mean function \( g \) in an inhomogeneous medium.}\label{table4}
\begin{tabular}{c | c| c | c| c}
\hline
 Algorithm & PCA  &  DMD&  U-Net  &pix2pix \\
\hline
\hline
Time &4 &6 &  79 &  697 \\
\hline
Relative error &0.74&0.64 & 0.27 &   0.28\\
\hline
\end{tabular}
\end{table}

\begin{table}[t]
  \centering
  \caption{Training time and $L^1$ relative error for the reconstruction of the variance function \( h \) in an inhomogeneous medium.}\label{table5}
\begin{tabular}{c | c| c | c| c}
\hline
 Algorithm & PCA  & DMD &  U-Net  &pix2pix \\
\hline
\hline
Time &4 & 6 &    104& 694 \\
\hline
Relative error &1.09 &0.75  &     0.30 & 0.36 \\
\hline
\end{tabular}
\end{table}

\begin{exam}
In this example, the methods are trained with noisy measurements to reconstruct the mean of the random source when medium is inhomogeneous.
\end{exam}

It can be observed from Table \ref{table6} that the reconstructions obtained through all the employed techniques exhibit considerable stability with respect to measurement noise. The next example is to address the generalization ability of our methods.

\begin{table}[t]
  \centering
  \caption{  $L^1$ relative error of  the mean reconstruction  with different methods trained with measurement data of different noise levels in an unknown inhomogeneous medium }\label{table6}
\begin{tabular}{c | c| c | c| c}
\hline
Noise level& PCA  & DMD &  U-Net  &pix2pix \\
\hline
\hline
 0.05 & 0.64 & 0.77 &  0.26  & 0.27 \\
 \hline
  0.1 & 0.64&  0.73&   0.27   & 0.27 \\
  \hline 
   0.2 & 0.65&  0.73&   0.27&  0.27\\
   \hline
\end{tabular}
\end{table}

\begin{exam}
In this example, we test the performance of the methods trained with the Cylinder dataset on several samples that are distinct from the training set. The medium is homogeneous and the mean reconstruction is presented. The first sample includes five inclusions and the second sample is a simplified phantom figure.
\end{exam}

Although the test samples differ from the samples in the training set, with more inclusions or inclusions of different shapes, the U-Net and pix2pix methods can still provide satisfactory reconstructions of well-separated inclusions with accurate positions, as shown in Figure \ref{Ex5.main_general}. All machine learning methods improve the reconstruction to varying extents. Compared to the other methods, the reconstructions of the pix2pix algorithm are less blurry, with the edges clearly reconstructed.

\begin{figure}[t]
\centering
\includegraphics[width=1\textwidth]{ 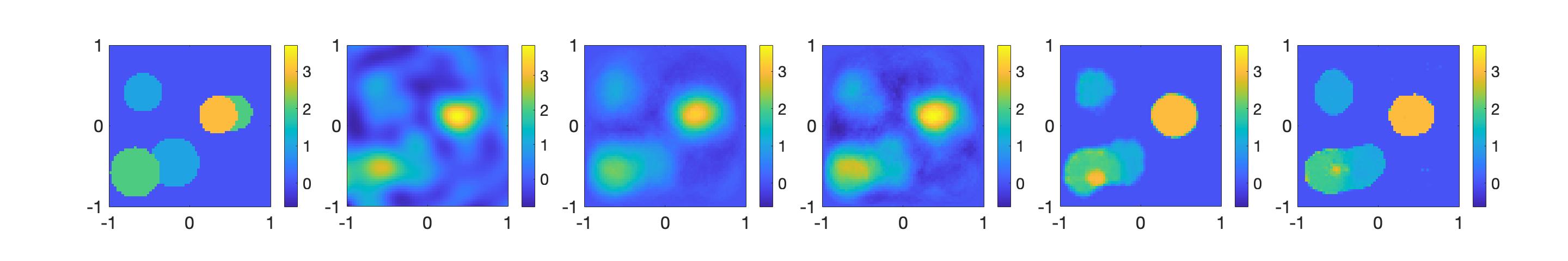} \vskip -1em 
\includegraphics[width=1\textwidth]{ 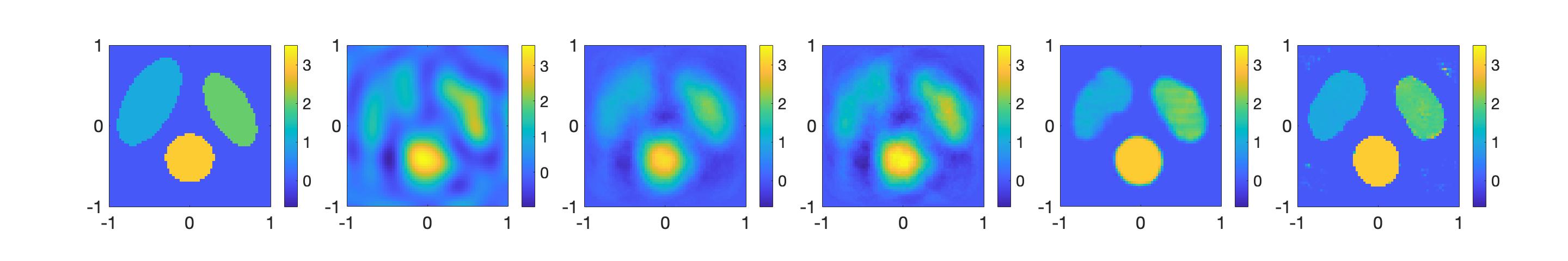} \vskip -1em 
\caption{Reconstructed mean of the random source of different shapes with the neural network trained with Cylinder dataset.  The columns from the left are the
ground-truth images and reconstructed images by the Kaczmarz method, PCA-based method, DMD-based method, U-Net based method, and pix2pix method.} \label{Ex5.main_general}
\end{figure}

\section{conclusion}

In this work, we have presented a novel approach to solving the inverse random source problem, a fundamental challenge in many scientific and engineering applications. Our two-stage method leverages data-assisted techniques to provide accurate and efficient solutions. In the first stage, we utilize the regularized Kaczmarz algorithm to obtain an approximation of the statistical properties of the random source based on integral equations. This approximation serves as the input to the second stage, where data-assisted methods are used to learn the mapping from the approximations to the exact mean and variance of the source.

To evaluate the effectiveness of our approach, we conduct a comparative study of different data-assisted approaches. The results demonstrate that neural network-assisted techniques outperform other methods, offering stable and accurate reconstructions with a relatively small number of realizations for both homogeneous and inhomogeneous media.

\end{document}